\documentclass[12pt,a4paper]{amsart}
\usepackage{amsfonts}
\usepackage{amsthm}
\usepackage{amsmath}
\usepackage{amscd}
\usepackage[latin2]{inputenc}
\usepackage{t1enc}
\usepackage[mathscr]{eucal}
\usepackage{indentfirst}
\usepackage{graphicx}
\usepackage{graphics}
\usepackage{pict2e}
\usepackage{epic}
\numberwithin{equation}{section}
\usepackage[margin=2.9cm]{geometry}
\usepackage{epstopdf}

\theoremstyle{plain}
\newtheorem{Th}{Theorem}[section]
\newtheorem{Lemma}[Th]{Lemma}
\newtheorem{Cor}[Th]{Corollary}
\newtheorem{Prop}[Th]{Proposition}

 \theoremstyle{definition}
\newtheorem{Def}[Th]{Definition}

\newtheorem{Rem}[Th]{Remark}
\newtheorem{?}[Th]{Problem}
\newtheorem{Ex}[Th]{Example}
\newtheorem{Exs}[Th]{Examples}

\begin{document}

\title{some generalizations of frames in Hilbert modules}

\author[Mohamed Rossafi]{Mohamed Rossafi}

\address{%
	Department of Mathematics\\
	University of Ibn Tofail\\
	B.P. 133, Kenitra \\
	Morocco}

\email{rossafimohamed@gmail.com}

\author{Samir Kabbaj}
\address{Department of Mathematics\\
	University of Ibn Tofail\\
	B.P. 133, Kenitra \\
	Morocco}
\email{samkabbaj@yahoo.fr}

 \subjclass[2010]{Primary: 42C15; Secondary: 47C15}

 \keywords{Frame, generalized frame, $C^{\ast}$-algebra, Hilbert $C^{\ast}$-modules.}

\begin{abstract} Frames play significant role in various areas of science and engineering. In this paper, we introduce the concepts of frames for $End_{\mathcal{A}}^{\ast}(\mathcal{H, K})$ and their generalizations. Moreover, we obtain some new results for generalized frames in Hilbert modules.
\end{abstract}

\maketitle

\section{Introduction and preliminaries} 

The concept of frames in Hilbert spaces has been introduced by
Duffin and Schaeffer \cite{Duf} in 1952 to study some deep problems in nonharmonic Fourier
series. After the fundamental paper \cite{13} by Daubechies, Grossman and Meyer, frame
theory began to be widely used, particularly in the more specialized context of wavelet
frames and Gabor frames \cite{Gab}.

Traditionally, frames have been used in signal processing, image processing, data compression
and sampling theory. A discreet frame is a countable family of
elements in a separable Hilbert space which allows for a stable, not necessarily unique,
decomposition of an arbitrary element into an expansion of the frame elements.

In this paper, we introduce the concepts of frames for $End_{\mathcal{A}}^{\ast}(\mathcal{H, K})$ and their generalizations. Also we prove some new properties and we give some examples.

Let $I$ be a finite or countable index subset of $\mathbb{N}$. In this section we briefly recall the definitions and basic properties of $C^{\ast}$-algebra, Hilbert $\mathcal{A}$-modules, Frames in Hilbert $\mathcal{A}$-modules and their generalizations. For information about frames in Hilbert spaces we refer to \cite{Ch}. Our reference for $C^{\ast}$-algebras is \cite{{Dav}, {Con}}. For a $C^{\ast}$-algebra $\mathcal{A}$ if $a\in\mathcal{A}$ is positive we write $a\geq 0$ and $\mathcal{A}^{+}$ denotes the closed cone of positive elements in $\mathcal{A}$.
\begin{Def}\cite{Con}.
	If $\mathcal{A}$ is a Banach algebra, an involution is a map $ a\rightarrow a^{\ast} $ of $\mathcal{A}$ into itself such that for all $a$ and $b$ in $\mathcal{A}$ and all scalars $\alpha$ the following conditions hold:
	\begin{enumerate}
		\item  $(a^{\ast})^{\ast}=a$.
		\item  $(ab)^{\ast}=b^{\ast}a^{\ast}$.
		\item  $(\alpha a+b)^{\ast}=\bar{\alpha}a^{\ast}+b^{\ast}$.
	\end{enumerate}
\end{Def}
\begin{Def}\cite{Con}.
	A $\mathcal{C}^{\ast}$-algebra $\mathcal{A}$ is a Banach algebra with involution such that :$$\|a^{\ast}a\|=\|a\|^{2}$$ for every $a$ in $\mathcal{A}$.
\end{Def}
\begin{Exs}
	\begin{enumerate}
		\item	$ B(\mathcal{H}) $	the algebra of bounded operators on a Hilbert space $\mathcal{H}$, is a  $\mathcal{C}^{\ast}$-algebra, where for each operator $A$, $A^{\ast}$ is the adjoint of $A$.
		\item $C(X)$ the algebra of continuous functions on a compact space $X$, is an abelian $\mathcal{C}^{\ast}$-algebra, where $f^{\ast}(x):=\overline{f(x)}$.
		\item $C_{0}(X)$ the algebra of continuous functions on a locally compact space $X$ that vanish at infinity is an abelian $\mathcal{C}^{\ast}$-algebra, where $f^{\ast}(x):=\overline{f(x)}$.
	\end{enumerate}	
\end{Exs}
\begin{Def}\cite{Kap}.
	Let $ \mathcal{A} $ be a unital $C^{\ast}$-algebra and $\mathcal{H}$ be a left $ \mathcal{A} $-module, such that the linear structures of $\mathcal{A}$ and $ \mathcal{H} $ are compatible. $\mathcal{H}$ is a pre-Hilbert $\mathcal{A}$-module if $\mathcal{H}$ is equipped with an $\mathcal{A}$-valued inner product $\langle.,.\rangle :\mathcal{H}\times\mathcal{H}\rightarrow\mathcal{A}$, such that is sesquilinear, positive definite and respects the module action. In the other words,
	\begin{enumerate}
		\item $ \langle x,x\rangle\geq0 $ for all $ x\in\mathcal{H} $ and $ \langle x,x\rangle=0$ if and only if $x=0$.
		\item $\langle ax+y,z\rangle=a\langle x,y\rangle+\langle y,z\rangle$ for all $a\in\mathcal{A}$ and $x,y,z\in\mathcal{H}$.
		\item $ \langle x,y\rangle=\langle y,x\rangle^{\ast} $ for all $x,y\in\mathcal{H}$.
	\end{enumerate}	 
\end{Def}
For $x\in\mathcal{H}, $ we define $||x||=||\langle x,x\rangle||^{\frac{1}{2}}$. If $\mathcal{H}$ is complete with $||.||$, it is called a Hilbert $\mathcal{A}$-module or a Hilbert $C^{\ast}$-module over $\mathcal{A}$. For every $a$ in $C^{\ast}$-algebra $\mathcal{A}$, we have $|a|=(a^{\ast}a)^{\frac{1}{2}}$ and the $\mathcal{A}$-valued norm on $\mathcal{H}$ is defined by $|x|=\langle x, x\rangle^{\frac{1}{2}}$ for $x\in\mathcal{H}$.
\begin{Exs}
	\begin{enumerate}
		
		\item Let $X$ be a locally compact Hausdorff space and $\mathcal{H}$ a Hilbert space, the Banach space $C_{0}(X, \mathcal{H})$ of all continuous $\mathcal{H}$-valued functions vanishing at infinity is a Hilbert $C^{\ast}$-module over the $C^{\ast}$-algebra $C_{0}(X)$ with inner product $\langle f, g\rangle(x):= \langle f(x), g(x)\rangle$ and module operation $(\phi f)(x)= \phi(x)f(x)$, for all $\phi\in C_{0}(X)$ and $f\in C_{0}(X, \mathcal{H})$.
		\item If $ \{\mathcal{H}_{k}\}_{k\in\mathbf{N}} $ is a countable set of Hilbert $\mathcal{A}$-modules, then one can define their direct sum $ \oplus_{k\in\mathbb{N}}\mathcal{H}_{k} $. On the $\mathcal{A}$-module $ \oplus_{k\in\mathbb{N}}\mathcal{H}_{k} $ of all sequences $x=(x_{k})_{k\in\mathbb{N}}: x_{k}\in\mathcal{H}_{k}$, such that the series $ \sum_{k\in\mathbb{N}}\langle x_{k}, x_{k}\rangle_{\mathcal{A}} $ is norm-convergent in the $\mathcal{C}^{\ast}$-algebra $\mathcal{A}$, we define the inner product by
		\begin{equation*}
		\langle x, y\rangle:=\sum_{k\in\mathbb{N}}\langle x_{k}, y_{k}\rangle_{\mathcal{A}} 
		\end{equation*}
		for $x, y\in\oplus_{k\in\mathbb{N}}\mathcal{H}_{k} $.
		
		Then $\oplus_{k\in\mathbb{N}}\mathcal{H}_{k}$ is a Hilbert $\mathcal{A}$-module.
		
		The direct sum of a countable number of copies of a Hilbert $\mathcal{C}^{\ast}$-module $\mathcal{H}$ is denoted by $l^{2}(\mathcal{H})$.
	\end{enumerate}
\end{Exs}
Let $\mathcal{H}$ and $\mathcal{K}$ be two Hilbert $\mathcal{A}$-modules, A map $T:\mathcal{H}\rightarrow\mathcal{K}$ is said to be adjointable if there exists a map $T^{\ast}:\mathcal{K}\rightarrow\mathcal{H}$ such that $\langle Tx,y\rangle_{\mathcal{A}}=\langle x,T^{\ast}y\rangle_{\mathcal{A}}$ for all $x\in\mathcal{H}$ and $y\in\mathcal{K}$.

We also reserve the notation $End_{\mathcal{A}}^{\ast}(\mathcal{H},\mathcal{K})$ for the set of all adjointable operators from $\mathcal{H}$ to $\mathcal{K}$ and $End_{\mathcal{A}}^{\ast}(\mathcal{H},\mathcal{H})$ is abbreviated to $End_{\mathcal{A}}^{\ast}(\mathcal{H})$.

$End_{\mathcal{A}}^{\ast}(\mathcal{H},\mathcal{K})$ is a Hilbert $End_{\mathcal{A}}^{\ast}(\mathcal{K})$-module  with the inner product $\langle T, S\rangle=TS^{\ast}$, $\forall T, S\in B(\mathcal{H,K})$.
\begin{Lemma}\label{2.8}
	\cite{Ali}. Let $\mathcal{H}$ and $\mathcal{K}$ be two Hilbert $\mathcal{A}$-modules and $T\in End_{\mathcal{A}}^{\ast}(\mathcal{H},\mathcal{K})$.
	\begin{itemize}
		\item [(i)] If $T$ is injective and $T$ has closed range, then the adjointable map $T^{\ast}T$ is invertible and $$\|(T^{\ast}T)^{-1}\|^{-1}I_{\mathcal{H}}\leq T^{\ast}T\leq\|T\|^{2}I_{\mathcal{H}}.$$
		\item  [(ii)]	If $T$ is surjective, then the adjointable map $TT^{\ast}$ is invertible and $$\|(TT^{\ast})^{-1}\|^{-1}I_{\mathcal{K}}\leq TT^{\ast}\leq\|T\|^{2}I_{\mathcal{K}}.$$
	\end{itemize}	
\end{Lemma}
\begin{Def}
	\cite{AB}. We call a sequence $\{\Lambda_{i}\in End_{\mathcal{A}}^{\ast}(\mathcal{H},V_{i}):i\in I \}$ a g-frame in Hilbert $\mathcal{A}$-module $\mathcal{H}$ with respect to $\{V_{i}:i\in I \}$ if there exist two positive constants $C$, $D$, such that for all $x\in\mathcal{H}$, 
	\begin{equation}\label{2}
	C\langle x,x\rangle_{\mathcal{A}}\leq\sum_{i\in I}\langle \Lambda_{i}x,\Lambda_{i}x\rangle_{\mathcal{A}}\leq D\langle x,x\rangle_{\mathcal{A}}.
	\end{equation}
	The numbers $C$ and $D$ are called lower and upper bounds of the g-frame, respectively. If $C=D=\lambda$, the g-frame is $\lambda$-tight. If $C = D = 1$, it is called a g-Parseval frame. If the sum in the middle of \eqref{2} is convergent in norm, the g-frame is called standard.
\end{Def}
\begin{Def}
	\cite{Ali}. Let $ \mathcal{H} $ be a Hilbert $\mathcal{A}$-module over a unital $C^{\ast}$-algebra. A family $\{x_{i}\}_{i\in I}$ of elements of $\mathcal{H}$ is an $\ast$-frame for $ \mathcal{H} $, if there	exist strictly nonzero elements $A$, $B$ in $\mathcal{A}$, such that for all $x\in\mathcal{H}$,
	\begin{equation}\label{3}
	A\langle x,x\rangle_{\mathcal{A}} A^{\ast}\leq\sum_{i\in I}\langle x,x_{i}\rangle_{\mathcal{A}}\langle x_{i},x\rangle_{\mathcal{A}}\leq B\langle x,x\rangle_{\mathcal{A}} B^{\ast}.
	\end{equation}
	The elements $A$ and $B$ are called lower and upper bounds of the $\ast$-frame, respectively. If $A=B=\lambda_{1}$, the $\ast$-frame is $\lambda_{1}$-tight. If $A = B = 1$, it is called a normalized tight $\ast$-frame or a Parseval $\ast$-frame. If the sum in the middle of \eqref{3} is convergent in norm, the $\ast$-frame is called standard.
\end{Def}
\begin{Def}
	\cite{A}. We call a sequence $\{\Lambda_{i}\in End_{\mathcal{A}}^{\ast}(\mathcal{H},V_{i}):i\in I \}$ an $\ast$-g-frame in Hilbert $\mathcal{A}$-module $\mathcal{H}$ over a unital $C^{\ast}$-algebra with respect to $\{V_{i}:i\in I \}$ if there exist strictly nonzero elements $A$, $B$ in $\mathcal{A}$, such that for all $x\in\mathcal{H}$, 
	\begin{equation}\label{4}
	A\langle x,x\rangle_{\mathcal{A}} A^{\ast}\leq\sum_{i\in I}\langle \Lambda_{i}x,\Lambda_{i}x\rangle_{\mathcal{A}}\leq B\langle x,x\rangle_{\mathcal{A}} B^{\ast}.
	\end{equation}
	The elements $A$ and $B$ are called lower and upper bounds of the $\ast$-g-frame, respectively. If $A=B=\lambda_{1}$, the $\ast$-g-frame is $\lambda_{1}$-tight. If $A= B = 1$, it is called an $\ast$-g-Parseval frame. If the sum in the middle of \eqref{4} is convergent in norm, the $\ast$-g-frame is called standard.
\end{Def}
\begin{Def} \cite{Xiang}
	Let $K\in End_{\mathcal{A}}^{\ast}(\mathcal{H})$ and $\Lambda_{i}\in End_{\mathcal{A}}^{\ast}(\mathcal{H},V_{i})$ for all $i\in I$, then $\{\Lambda_{i}\}_{i\in I}$ is said to be a $K$-g-frame for $\mathcal{H}$ with respect to $\{V_{i}\}_{i\in I}$ if there exist two constants $C, D > 0$ such that
	\begin{equation}
	C\langle K^{\ast}x,K^{\ast}x\rangle_{\mathcal{A}}\leq\sum_{i\in I}\langle \Lambda_{i}x,\Lambda_{i}x\rangle_{\mathcal{A}}\leq D\langle x,x\rangle_{\mathcal{A}}, \forall x\in\mathcal{H}.
	\end{equation}
	The numbers $C$ and $D$ are called $K$-g-frame bounds. Particularly, if 
	\begin{equation}
	C\langle K^{\ast}x,K^{\ast}x\rangle =\sum_{i\in I}\langle \Lambda_{i}x,\Lambda_{i}x\rangle, \forall x\in\mathcal{H}.
	\end{equation}
	The $K$-g-frame is $C$-tight.
\end{Def}
\begin{Def} \cite{Dast}
	Let $K\in End_{\mathcal{A}}^{\ast}(\mathcal{H})$. A family $\{x_{i}\}_{i\in I}$ of elements of $\mathcal{H}$ is an $\ast$-K-frame for $ \mathcal{H} $, if there	exist strictly nonzero elements $A$, $B$ in $\mathcal{A}$, such that for all $x\in\mathcal{H}$,
	\begin{equation}
	A\langle K^{\ast}x,K^{\ast}x\rangle_{\mathcal{A}}A^{\ast}\leq\sum_{i\in I}\langle x,x_{i}\rangle_{\mathcal{A}}\langle x_{i},x\rangle_{\mathcal{A}}\leq B\langle x,x\rangle_{\mathcal{A}}B^{\ast}.
	\end{equation}
	The elements $A$ and $B$ are called lower and upper bound of the $\ast$-$K$-frame, respectively.
\end{Def}
\begin{Def} \label{our frame} \cite{Ros2}
	Let $K\in End_{\mathcal{A}}^{\ast}(\mathcal{H})$. We call a sequence $\{\Lambda_{i}\in End_{A}^{\ast}(\mathcal{H},\mathcal{H}_{i}):i\in I \}$ a $\ast$-K-g-frame in Hilbert $\mathcal{A}$-module $\mathcal{H}$ with respect to $\{\mathcal{H}_{i}:i\in I \}$ if there exist strictly nonzero elements $A$, $B$ in $\mathcal{A}$ such that 
	\begin{equation}
	A\langle K^{\ast}x,K^{\ast}x\rangle_{\mathcal{A}} A^{\ast}\leq\sum_{i\in I}\langle \Lambda_{i}x,\Lambda_{i}x\rangle_{\mathcal{A}}\leq B\langle x,x\rangle_{\mathcal{A}} B^{\ast}, \forall x\in\mathcal{H}.
	\end{equation}
	The numbers $A$ and $B$ are called lower and upper bounds of the $\ast$-K-g-frame, respectively. If 
	\begin{equation}
	A\langle K^{\ast}x,K^{\ast}x\rangle A^{\ast}=\sum_{i\in I}\langle \Lambda_{i}x,\Lambda_{i}x\rangle, \forall x\in\mathcal{H}.
	\end{equation}
	The $\ast$-K-g-frame is $A$-tight.
\end{Def}
\section{some results for generalized frames in Hilbert $\mathcal{A}$-modules}
We begin this section with the follwing Theorem.
\begin{Th}
Let $ \mathcal{H} $ be a Hilbert $\mathcal{A}$-module over a commutative $C^{\ast}$-algebra. A family $\{x_{i}\}_{i\in I}$ of elements of $\mathcal{H}$ is an $\ast$-frame for $ \mathcal{H} $, if and only if there	exist strictly positive elements $A$ and $B$ in $\mathcal{A}$, such that for all $x\in\mathcal{H}$,
\begin{equation}
A\langle x,x\rangle_{\mathcal{A}}\leq\sum_{i\in I}\langle x,x_{i}\rangle_{\mathcal{A}}\langle x_{i},x\rangle_{\mathcal{A}}\leq B\langle x,x\rangle_{\mathcal{A}}.
\end{equation}
\end{Th}
\begin{proof}
	Let $\{x_{i}\}_{i\in I}$ be an $\ast$-frame for $ \mathcal{H} $, then there	exist strictly nonzero elements $a$ and $b$ in $\mathcal{A}$, such that for all $x\in\mathcal{H}$,
	\begin{equation*}
	a\langle x,x\rangle_{\mathcal{A}}a^{\ast}\leq\sum_{i\in I}\langle x,x_{i}\rangle_{\mathcal{A}}\langle x_{i},x\rangle_{\mathcal{A}}\leq b\langle x,x\rangle_{\mathcal{A}}b^{\ast}.
	\end{equation*}
	By commutativity we have
	\begin{equation*}
	aa^{\ast}\langle x,x\rangle_{\mathcal{A}}\leq\sum_{i\in I}\langle x,x_{i}\rangle_{\mathcal{A}}\langle x_{i},x\rangle_{\mathcal{A}}\leq bb^{\ast}\langle x,x\rangle_{\mathcal{A}}.
	\end{equation*}
We pose $A=aa^{\ast}$ and $B=bb^{\ast}$, then $A$ and $B$ are strictly positive elements in $\mathcal{A}$.
Hence 
\begin{equation*}
A\langle x,x\rangle_{\mathcal{A}}\leq\sum_{i\in I}\langle x,x_{i}\rangle_{\mathcal{A}}\langle x_{i},x\rangle_{\mathcal{A}}\leq B\langle x,x\rangle_{\mathcal{A}}.
\end{equation*}
Conversely, let $A$ and $B$ are strictly positive elements in $\mathcal{A}$, such that for all $x\in\mathcal{H}$
\begin{equation*}
A\langle x,x\rangle_{\mathcal{A}}\leq\sum_{i\in I}\langle x,x_{i}\rangle_{\mathcal{A}}\langle x_{i},x\rangle_{\mathcal{A}}\leq B\langle x,x\rangle_{\mathcal{A}}.
\end{equation*}
$A$ and $B$ are strictly positive elements in $\mathcal{A}$ then there exist $a$ and $b$ in $\mathcal{A}$ such that $A=aa^{\ast}$ and $B=bb^{\ast}$.

So
\begin{equation*}
aa^{\ast}\langle x,x\rangle_{\mathcal{A}}\leq\sum_{i\in I}\langle x,x_{i}\rangle_{\mathcal{A}}\langle x_{i},x\rangle_{\mathcal{A}}\leq bb^{\ast}\langle x,x\rangle_{\mathcal{A}}.
\end{equation*}
By commutativity we have
\begin{equation*}
a\langle x,x\rangle_{\mathcal{A}}a^{\ast}\leq\sum_{i\in I}\langle x,x_{i}\rangle_{\mathcal{A}}\langle x_{i},x\rangle_{\mathcal{A}}\leq b\langle x,x\rangle_{\mathcal{A}}b^{\ast}.
\end{equation*}
Therefore $\{x_{i}\}_{i\in I}$ is an $\ast$-frame for $ \mathcal{H} $.
\end{proof}
\begin{Cor}
A sequence $\{\Lambda_{i}\in End_{\mathcal{A}}^{\ast}(\mathcal{H},V_{i}):i\in I \}$ is an $\ast$-g-frame in Hilbert $\mathcal{A}$-module $\mathcal{H}$ over a commutative $C^{\ast}$-algebra with respect to $\{V_{i}:i\in I \}$ if and only if there	exist strictly positive elements $A$, $B$ in $\mathcal{A}$, such that for all $x\in\mathcal{H}$,
\begin{equation}
A\langle x,x\rangle_{\mathcal{A}}\leq\sum_{i\in I}\langle \Lambda_{i}x,\Lambda_{i}x\rangle_{\mathcal{A}}\leq B\langle x,x\rangle_{\mathcal{A}}.
\end{equation}
\end{Cor}
\begin{Cor}
Let $K\in End_{\mathcal{A}}^{\ast}(\mathcal{H})$. A family $\{x_{i}\}_{i\in I}$ of elements of $\mathcal{H}$ is an $\ast$-K-frame for $ \mathcal{H} $, if and only if there	exist strictly positive elements $A$, $B$ in $\mathcal{A}$, such that for all $x\in\mathcal{H}$,
\begin{equation}
A\langle K^{\ast}x,K^{\ast}x\rangle_{\mathcal{A}}\leq\sum_{i\in I}\langle x,x_{i}\rangle_{\mathcal{A}}\langle x_{i},x\rangle_{\mathcal{A}}\leq B\langle x,x\rangle_{\mathcal{A}}.
\end{equation}	
\end{Cor}
\begin{Cor}
	Let $K\in End_{\mathcal{A}}^{\ast}(\mathcal{H})$. A sequence $\{\Lambda_{i}\in End_{A}^{\ast}(\mathcal{H},\mathcal{H}_{i}):i\in I \}$ is an $\ast$-K-g-frame in Hilbert $\mathcal{A}$-module $\mathcal{H}$ over a commutative $C^{\ast}$-algebra, with respect to $\{\mathcal{H}_{i}:i\in I \}$ if and only if there exist strictly positive elements $A$, $B$ in $\mathcal{A}$, such that for all $x\in\mathcal{H}$, 
\begin{equation}
A\langle K^{\ast}x,K^{\ast}x\rangle_{\mathcal{A}}\leq\sum_{i\in I}\langle \Lambda_{i}x,\Lambda_{i}x\rangle_{\mathcal{A}}\leq B\langle x,x\rangle_{\mathcal{A}}, \forall x\in\mathcal{H}.
\end{equation}	
\end{Cor}
For a sequence of Hilbert $\mathcal{A}$-module  $\{\mathcal{K}_{i}\}_{i\in I}$, the space $\oplus_{i\in I}\mathcal{K}_{i}$ is a Hilbert $\mathcal{A}$-module with the inner product
\begin{equation*}
\Big\langle\{x_{i}\}_{i\in I}, \{y_{i}\}_{i\in I} \Big\rangle=\sum_{i\in I}\langle x_{i}, y_{i}\rangle_{\mathcal{A}},
\end{equation*}

\begin{Prop}
	Let $\mathcal{K}=\oplus_{i\in I}\mathcal{K}_{i}$. 
	\begin{enumerate}
		\item 	The sequence $\{\Lambda_{i}\in End_{\mathcal{A}}^{\ast}(\mathcal{H, K}_{i}) : i\in I\}$ is a g-frame for $\mathcal{H}$ with respect to $\{\mathcal{K}_{i}\}_{i\in I}$ if and only if the sequence $\{\tilde{\Lambda_{i}}\in End_{\mathcal{A}}^{\ast}(\mathcal{H, K}) : i\in I\}$ is a g-frame for $\mathcal{H}$ with respect to $\mathcal{K}$, with $\tilde{\Lambda_{i}}x=(..., 0, 0, \Lambda_{i}x, 0, 0,...), \forall x\in\mathcal{H}$.
		\item The sequence $\{\Lambda_{i}\in End_{\mathcal{A}}^{\ast}(\mathcal{H, K}_{i}) : i\in I\}$ is an $\ast$-g-frame for $\mathcal{H}$ with respect to $\{\mathcal{K}_{i}\}_{i\in I}$ if and only if the sequence $\{\tilde{\Lambda_{i}}\in End_{\mathcal{A}}^{\ast}(\mathcal{H, K}) : i\in I\}$ is an $\ast$-g-frame for $\mathcal{H}$ with respect to $\mathcal{K}$, with $\tilde{\Lambda_{i}}x=(..., 0, 0, \Lambda_{i}x, 0, 0,...), \forall x\in\mathcal{H}$.
		\item For $K\in End_{\mathcal{A}}^{\ast}(\mathcal{H})$, the sequence $\{\Lambda_{i}\in End_{\mathcal{A}}^{\ast}(\mathcal{H, K}_{i}) : i\in I\}$ is a $K$-g-frame for $\mathcal{H}$ with respect to $\{\mathcal{K}_{i}\}_{i\in I}$ if and only if the sequence $\{\tilde{\Lambda_{i}}\in End_{\mathcal{A}}^{\ast}(\mathcal{H, K}) : i\in I\}$ is a $K$-g-frame for $\mathcal{H}$ with respect to $\mathcal{K}$, with $\tilde{\Lambda_{i}}x=(..., 0, 0, \Lambda_{i}x, 0, 0,...), \forall x\in\mathcal{H}$.
		\item For $K\in End_{\mathcal{A}}^{\ast}(\mathcal{H})$, the sequence $\{\Lambda_{i}\in End_{\mathcal{A}}^{\ast}(\mathcal{H, K}_{i}) : i\in I\}$ is an $\ast$-$K$-g-frame for $\mathcal{H}$ with respect to $\{\mathcal{K}_{i}\}_{i\in I}$ if and only if the sequence $\{\tilde{\Lambda_{i}}\in End_{\mathcal{A}}^{\ast}(\mathcal{H, K}) : i\in I\}$ is an $\ast$-$K$-g-frame for $\mathcal{H}$ with respect to $\mathcal{K}$, with $\tilde{\Lambda_{i}}x=(..., 0, 0, \Lambda_{i}x, 0, 0,...), \forall x\in\mathcal{H}$.
\end{enumerate}
\end{Prop}
\section{frame for $End_{\mathcal{A}}^{\ast}(\mathcal{H, K})$}
We begin this section with the follwing definition.
\begin{Def}
	A sequence $\{T_{i}\in End_{\mathcal{A}}^{\ast}(\mathcal{H, K}): i\in I\}$ is said to be a frame for $End_{A}^{\ast}(\mathcal{H, K})$ if there exist $0 < A, B < \infty$ sach that
	\begin{equation}\label{eq3.}
	A\langle T,T\rangle\leq\sum_{i\in I}\langle T,T_{i}\rangle\langle T_{i},T\rangle\leq B\langle T,T\rangle, \forall T\in End_{\mathcal{A}}^{\ast}(\mathcal{H, K}),
	\end{equation}
	where the series converges in the strong operator topology.
\end{Def}
\begin{Ex}
	For $n\in\mathbf{N^{\ast}}$, let $\{T_{i}\}_{i=1}^{n}\subset End_{\mathcal{A}}^{\ast}(\mathcal{H, K})$ such that for all $i=1,...,n$ $T_{i}$ is injective and has a closed range. 
	
	Then for all $i=1,...,n$, we have
	\begin{equation*}
\|(T_{i}^{\ast}T_{i})^{-1}\|^{-1}I_{\mathcal{H}}\leq T_{i}^{\ast}T_{i}\leq\|T_{i}\|^{2}I_{\mathcal{H}}.
	\end{equation*} 
	So 
	\begin{equation*}
	\bigg(\sum_{i=1}^{n}\|(T_{i}^{\ast}T_{i})^{-1}\|^{-1}\bigg)I_{\mathcal{H}}\leq\sum_{i=1}^{n} T_{i}^{\ast}T_{i}\leq\bigg(\sum_{i=1}^{n}\|T_{i}\|^{2}\bigg)I_{\mathcal{H}}.
	\end{equation*} 
	Hence
	\begin{equation*}
	\bigg(\sum_{i=1}^{n}\|(T_{i}^{\ast}T_{i})^{-1}\|^{-1}\bigg)TT^{\ast}\leq\sum_{i=1}^{n}T T_{i}^{\ast}T_{i}T^{\ast}\leq\bigg(\sum_{i=1}^{n}\|T_{i}\|^{2}\bigg)TT^{\ast}, \forall T\in End_{\mathcal{A}}^{\ast}(\mathcal{H, K}).
	\end{equation*} 
	Thus
	\begin{equation*}
\bigg(\sum_{i=1}^{n}\|(T_{i}^{\ast}T_{i})^{-1}\|^{-1}\bigg)\langle T, T\rangle\leq\sum_{i=1}^{n}\langle T, T_{i}\rangle\langle T_{i}, T\rangle\leq\bigg(\sum_{i=1}^{n}\|T_{i}\|^{2}\bigg)\langle T, T\rangle, 
	\end{equation*}
	for all $T\in End_{\mathcal{A}}^{\ast}(\mathcal{H, K})$.
	
	This shows that $\{T_{i}\}_{i=1}^{n}$ is a frame for $End_{\mathcal{A}}^{\ast}(\mathcal{H, K})$.
\end{Ex}
\begin{Ex}
	Let $\{T_{i}\}_{i\in\mathbf{N^{\ast}}}\subset End_{\mathcal{A}}^{\ast}(\mathcal{H, K})$ such that for all $i\in\mathbf{N^{\ast}}$ $T_{i}$ is injective, has a closed range and $\|T_{i}\|\leq\frac{1}{i}$. 
	
	Then for all $i\in\mathbf{N^{\ast}}$, we have
	\begin{equation*}
	\|(T_{i}^{\ast}T_{i})^{-1}\|^{-1}I_{\mathcal{H}}\leq T_{i}^{\ast}T_{i}\leq\|T_{i}\|^{2}I_{\mathcal{H}}.
	\end{equation*} 
	So 
	\begin{equation*}
	\bigg(\sum_{i\in\mathbf{N^{\ast}}}\|(T_{i}^{\ast}T_{i})^{-1}\|^{-1}\bigg)I_{\mathcal{H}}\leq\sum_{i\in\mathbf{N^{\ast}}} T_{i}^{\ast}T_{i}\leq\bigg(\sum_{i\in\mathbf{N^{\ast}}}\|T_{i}\|^{2}\bigg)I_{\mathcal{H}}.
	\end{equation*} 
	Hence
	\begin{equation*}
	\bigg(\sum_{i\in\mathbf{N^{\ast}}}\|(T_{i}^{\ast}T_{i})^{-1}\|^{-1}\bigg)TT^{\ast}\leq\sum_{i\in\mathbf{N^{\ast}}}T T_{i}^{\ast}T_{i}T^{\ast}\leq\bigg(\sum_{i\in\mathbf{N^{\ast}}}\|T_{i}\|^{2}\bigg)TT^{\ast}, \forall T\in End_{\mathcal{A}}^{\ast}(\mathcal{H, K}).
	\end{equation*} 
	Thus
	\begin{equation*}
	\bigg(\sum_{i\in\mathbf{N^{\ast}}}\|(T_{i}^{\ast}T_{i})^{-1}\|^{-1}\bigg)\langle T, T\rangle\leq\sum_{i\in\mathbf{N^{\ast}}}\langle T, T_{i}\rangle\langle T_{i}, T\rangle\leq\bigg(\sum_{i\in\mathbf{N^{\ast}}}\|T_{i}\|^{2}\bigg)\langle T, T\rangle, 
	\end{equation*}
	for all $T\in End_{\mathcal{A}}^{\ast}(\mathcal{H, K})$.
	
	This shows that $\{T_{i}\}_{i\in\mathbf{N^{\ast}}}$ is a frame for $End_{\mathcal{A}}^{\ast}(\mathcal{H, K})$.
\end{Ex}
\begin{Th}
	A sequence $\{T_{i}\in End_{\mathcal{A}}^{\ast}(\mathcal{H, K}): i\in I\}$ is a frame for $End_{\mathcal{A}}^{\ast}(\mathcal{H, K})$ if and only if it is a g-frame for $\mathcal{H}$ with respect to $\mathcal{K}$.
\end{Th}
\begin{proof}
Let $\{T_{i}\}_{i\in I}$ be a g-frame for $\mathcal{H}$ with respect to $\mathcal{K}$.

Then there exists two positive constants $A$ and $B$, such that 
\begin{equation*}
A\langle x, x\rangle_{\mathcal{A}}\leq\sum_{i\in I}\langle T_{i}x, T_{i}x\rangle_{\mathcal{A}}\leq B\langle x, x\rangle_{\mathcal{A}}, \forall x\in\mathcal{H},
\end{equation*}
\begin{equation*}\Longleftrightarrow
A\langle x,x\rangle_{\mathcal{A}}\leq\sum_{i\in I}\langle T_{i}^{\ast}T_{i}x,x\rangle_{\mathcal{A}}\leq B\langle x,x\rangle_{\mathcal{A}}, \forall x\in\mathcal{H}.
\end{equation*}
So
\begin{equation*}
AI_{\mathcal{H}}\leq\sum_{i\in I}T_{i}^{\ast}T_{i}\leq BI_{\mathcal{H}}.
\end{equation*}
Hence
\begin{equation*}
ATT^{\ast}\leq\sum_{i\in I}TT_{i}^{\ast}T_{i}T^{\ast}\leq BTT^{\ast}, \forall T\in End_{\mathcal{A}}^{\ast}(\mathcal{H, K}).
\end{equation*}
Thus
\begin{equation*}
A\langle T,T\rangle\leq\sum_{i\in I}\langle T,T_{i}\rangle\langle T_{i},T\rangle\leq B\langle T,T\rangle, \forall T\in End_{\mathcal{A}}^{\ast}(\mathcal{H, K}),
\end{equation*}
i.e $\{T_{i}\}_{i\in I}$ is a frame for $End_{\mathcal{A}}^{\ast}(\mathcal{H, K})$.

Conversely, assume that $\{T_{i}\}_{i\in I}$ be a frame for $End_{\mathcal{A}}^{\ast}(\mathcal{H, K})$.

Then there exists two positive constants $A$ and $B$, such that
\begin{equation*}
A\langle T,T\rangle\leq\sum_{i\in I}\langle T,T_{i}\rangle\langle T_{i},T\rangle\leq B\langle T,T\rangle, \forall T\in End_{\mathcal{A}}^{\ast}(\mathcal{H, K}),
\end{equation*}

\begin{equation*}\Longleftrightarrow
ATT^{\ast}\leq\sum_{i\in I}TT_{i}^{\ast}T_{i}T^{\ast}\leq BTT^{\ast}, \forall T\in End_{\mathcal{A}}^{\ast}(\mathcal{H, K}).
\end{equation*}
So
\begin{equation*}
A\langle TT^{\ast}x,x\rangle_{\mathcal{A}}\leq\sum_{i\in I}\langle TT_{i}^{\ast}T_{i}T^{\ast}x,x\rangle_{\mathcal{A}}\leq B\langle TT^{\ast}x,x\rangle_{\mathcal{A}}, \forall T\in End_{\mathcal{A}}^{\ast}(\mathcal{H, K}), \forall x\in\mathcal{K}.
\end{equation*}
Let $y\in\mathcal{H}$ and $T\in End_{\mathcal{A}}^{\ast}(\mathcal{H, K})$ such that $T^{\ast}x=y$, then
\begin{equation*}
A\langle y,y\rangle_{\mathcal{A}}\leq\sum_{i\in I}\langle T_{i}^{\ast}T_{i}y,y\rangle_{\mathcal{A}}\leq B\langle y,y\rangle_{\mathcal{A}}, \forall y\in\mathcal{H},
\end{equation*}
i.e
\begin{equation*}
A\langle y,y\rangle_{\mathcal{A}}\leq\sum_{i\in I}\langle T_{i}y,T_{i}y\rangle_{\mathcal{A}}\leq B\langle y,y\rangle_{\mathcal{A}}, \forall y\in\mathcal{H},
\end{equation*}
thus  $\{T_{i}\}_{i\in I}$ is a g-frame for $\mathcal{H}$ with respect to $\mathcal{K}$.
\end{proof}
\begin{Cor}
	A sequence $\{T_{i}\in End_{A}^{\ast}(\mathcal{H, K}): i\in I\}$ is a tight frame for $End_{A}^{\ast}(\mathcal{H, K})$ if and only if it is a tight g-frame for $\mathcal{H}$ whit respect to $\mathcal{K}$.
\end{Cor}
\begin{Cor}
	A sequence $\{T_{i}\in End_{A}^{\ast}(\mathcal{H, K}): i\in I\}$ is a Bessel sequence for $End_{A}^{\ast}(\mathcal{H, K})$ if and only if it is a g-Bessel sequence for $\mathcal{H}$ whit respect to $\mathcal{K}$.
\end{Cor}
\section{generalized frame for $End_{\mathcal{A}}^{\ast}(\mathcal{H, K})$}
We begin this section with the follwing definition.
\begin{Def}
	A sequence $\{T_{i}\in End_{\mathcal{A}}^{\ast}(\mathcal{H, K}): i\in I\}$ is said to be a generalized frame for $End_{\mathcal{A}}^{\ast}(\mathcal{H, K})$ if there exist strictly positive elements $A$, $B$ in $\mathcal{A}$ sach that
	\begin{equation}
	A\langle T,T\rangle\leq\sum_{i\in I}\langle T,T_{i}\rangle\langle T_{i},T\rangle\leq B\langle T,T\rangle, \forall T\in End_{\mathcal{A}}^{\ast}(\mathcal{H, K}),
	\end{equation}
	where the series converges in the strong operator topology.
\end{Def}
In the follwing $\mathcal{A}$ is a commutative $C^{\ast}$-algebra.
\begin{Th}
	A sequence $\{T_{i}\in End_{A}^{\ast}(\mathcal{H, K}): i\in I\}$ is a generalized frame for $End_{A}^{\ast}(\mathcal{H, K})$ if and only if it is an $\ast$-g-frame for $\mathcal{H}$ whit respect to $\mathcal{K}$.
\end{Th}
\begin{proof}
	Let $\{T_{i}\}_{i\in I}$ be an $\ast$-g-frame for $\mathcal{H}$ with respect to $\mathcal{K}$.
	
	Then there exist strictly positive elements $A$ and $B$ in $\mathcal{A}$, such that 
	\begin{equation*}
	A\langle x, x\rangle_{\mathcal{A}}\leq\sum_{i\in I}\langle T_{i}x, T_{i}x\rangle_{\mathcal{A}}\leq B\langle x, x\rangle_{\mathcal{A}}, \forall x\in\mathcal{H},
	\end{equation*}
	 \begin{equation*}\Longleftrightarrow
	A\langle x,x\rangle_{\mathcal{A}}\leq\sum_{i\in I}\langle T_{i}^{\ast}T_{i}x,x\rangle_{\mathcal{A}}\leq B\langle x,x\rangle_{\mathcal{A}}, \forall x\in\mathcal{H}.
	\end{equation*}
	So
	\begin{equation*}
	AI_{\mathcal{H}}\leq\sum_{i\in I}T_{i}^{\ast}T_{i}\leq BI_{\mathcal{H}}.
	\end{equation*}
	Hence
	\begin{equation*}
	ATT^{\ast}\leq\sum_{i\in I}TT_{i}^{\ast}T_{i}T^{\ast}\leq BTT^{\ast}, \forall T\in End_{\mathcal{A}}^{\ast}(\mathcal{H, K}).
	\end{equation*}
	Thus
	\begin{equation*}
	A\langle T,T\rangle\leq\sum_{i\in I}\langle T,T_{i}\rangle\langle T_{i},T\rangle\leq B\langle T,T\rangle, \forall T\in End_{\mathcal{A}}^{\ast}(\mathcal{H, K}),
	\end{equation*}
	i.e $\{T_{i}\}_{i\in I}$ is a generalized frame for $End_{\mathcal{A}}^{\ast}(\mathcal{H, K})$.
	
	Conversely, assume that $\{T_{i}\}_{i\in I}$ be a generalized frame for $End_{\mathcal{A}}^{\ast}(\mathcal{H, K})$.
	
	Then there exist strictly positive elements $A$ and $B$ in $\mathcal{A}$, such that
	\begin{equation*}
	A\langle T,T\rangle\leq\sum_{i\in I}\langle T,T_{i}\rangle\langle T_{i},T\rangle\leq B\langle T,T\rangle, \forall T\in End_{\mathcal{A}}^{\ast}(\mathcal{H, K}),
	\end{equation*}
	
	\begin{equation*}\Longleftrightarrow
	ATT^{\ast}\leq\sum_{i\in I}TT_{i}^{\ast}T_{i}T^{\ast}\leq BTT^{\ast}, \forall T\in End_{\mathcal{A}}^{\ast}(\mathcal{H, K}).
	\end{equation*}
	So
	\begin{equation*}
	A\langle TT^{\ast}x,x\rangle_{\mathcal{A}}\leq\sum_{i\in I}\langle TT_{i}^{\ast}T_{i}T^{\ast}x,x\rangle_{\mathcal{A}}\leq B\langle TT^{\ast}x,x\rangle_{\mathcal{A}}, \forall T\in End_{\mathcal{A}}^{\ast}(\mathcal{H, K}), \forall x\in\mathcal{K}.
	\end{equation*}
	Let $y\in\mathcal{H}$ and $T\in End_{\mathcal{A}}^{\ast}(\mathcal{H, K})$ such that $T^{\ast}x=y$, then
	\begin{equation*}
	A\langle y,y\rangle_{\mathcal{A}}\leq\sum_{i\in I}\langle T_{i}^{\ast}T_{i}y,y\rangle_{\mathcal{A}}\leq B\langle y,y\rangle_{\mathcal{A}}, \forall y\in\mathcal{H},
	\end{equation*}
	i.e
	\begin{equation*}
	A\langle y,y\rangle_{\mathcal{A}}\leq\sum_{i\in I}\langle T_{i}y,T_{i}y\rangle_{\mathcal{A}}\leq B\langle y,y\rangle_{\mathcal{A}}, \forall y\in\mathcal{H},
	\end{equation*}
	thus  $\{T_{i}\}_{i\in I}$ is an $\ast$-g-frame for $\mathcal{H}$ with respect to $\mathcal{K}$.
\end{proof}
\begin{Cor}
	A sequence $\{T_{i}\in End_{A}^{\ast}(\mathcal{H, K}): i\in I\}$ is a generalized tight frame for $End_{A}^{\ast}(\mathcal{H, K})$ if and only if it is a tight $\ast$-g-frame for $\mathcal{H}$ whit respect to $\mathcal{K}$.
\end{Cor}
\begin{Cor}
	A sequence $\{T_{i}\in End_{A}^{\ast}(\mathcal{H, K}): i\in I\}$ is a generalized Bessel sequence for $End_{A}^{\ast}(\mathcal{H, K})$ if and only if it is an $\ast$-g-Bessel sequence for $\mathcal{H}$ whit respect to $\mathcal{K}$.
\end{Cor}
\section{K-frame for $End_{A}^{\ast}(\mathcal{H, K})$}
We begin this section with the follwing definition.
\begin{Def}
	Let $K\in End_{A}^{\ast}(\mathcal{H})$. A sequence $\{T_{i}\in End_{A}^{\ast}(\mathcal{H, K}): i\in I\}$ is said to be a $K$-frame for $End_{A}^{\ast}(\mathcal{H, K})$ if there exist $0 < A, B < \infty$ sach that
	\begin{equation}\label{eq3.3.2}
	A\langle TK,TK\rangle\leq\sum_{i\in I}\langle T,T_{i}\rangle\langle T_{i},T\rangle\leq B\langle T,T\rangle, \forall T\in End_{A}^{\ast}(\mathcal{H,K}),
	\end{equation}
	where the series converges in the strong operator topology.
\end{Def}
\begin{Ex}
	For $n\in\mathbf{N^{\ast}}$, let $\{T_{i}\}_{i=1}^{n}\subset End_{\mathcal{A}}^{\ast}(\mathcal{H})$. 
	
	Then for $K=\big(\sum_{i=1}^{n} T_{i}^{\ast}T_{i}\big)^{\frac{1}{2}} $, we have
	\begin{equation*}
K^{\ast}K=\sum_{i=1}^{n} T_{i}^{\ast}T_{i}
	\end{equation*} 
	So 
	\begin{equation*}
	TKK^{\ast}T^{\ast}=\sum_{i=1}^{n} TT_{i}^{\ast}T_{i}T^{\ast}
	\end{equation*} 
	Hence
	\begin{equation*}
	\langle TK, TK\rangle=\sum_{i=1}^{n}\langle T, T_{i}\rangle\langle T_{i}, T\rangle,  \forall T\in End_{\mathcal{A}}^{\ast}(\mathcal{H}).
	\end{equation*} 
	
	This shows that $\{T_{i}\}_{i=1}^{n}$ is a Parseval $K$-frame for $End_{\mathcal{A}}^{\ast}(\mathcal{H})$.
\end{Ex}
	Similar to Remark $1$ in \cite{Ros2} we have 
\begin{Rem}
	\begin{enumerate}
		\item Every frame for $End_{A}^{\ast}(\mathcal{H,K})$ is a $K$-frame, for any $K\in End_{A}^{\ast}(\mathcal{H})$: $K\neq0$.
		\item If $K\in End_{A}^{\ast}(\mathcal{H})$ is a surjective operator, then every $K$-frame for $End_{A}^{\ast}(\mathcal{H,K})$ is a frame for $End_{A}^{\ast}(\mathcal{H, K})$.
	\end{enumerate}
\end{Rem}
\begin{Th}
	Let $K\in End_{\mathcal{A}}^{\ast}(\mathcal{H})$ and $\{T_{i}\in End_{\mathcal{A}}^{\ast}(\mathcal{H, K}): i\in I\}$ be a frame for $End_{\mathcal{A}}^{\ast}(\mathcal{H, K})$. Then $\{T_{i}K^{\ast}\}_{i\in I}$ is a $K$-frame for $End_{\mathcal{A}}^{\ast}(\mathcal{H, K})$.
\end{Th}
\begin{proof}
	Let $\{T_{i}\}_{i\in I}$ be a frame for $End_{\mathcal{A}}^{\ast}(\mathcal{H, K})$.
	
	Then there exists two positive constants $A$ and $B$, such that
	\begin{equation*}
	A\langle T,T\rangle\leq\sum_{i\in I}\langle T,T_{i}\rangle\langle T_{i},T\rangle\leq B\langle T,T\rangle, \forall T\in End_{\mathcal{A}}^{\ast}(\mathcal{H, K}).
	\end{equation*}
	Hence
	\begin{equation*}
	A\langle TK,TK\rangle\leq\sum_{i\in I}\langle TK,T_{i}\rangle\langle T_{i},TK\rangle\leq B\langle TK,TK\rangle, \forall T\in End_{\mathcal{A}}^{\ast}(\mathcal{H, K}).
	\end{equation*}
	Thus
	\begin{equation*}
	A\langle TK,TK\rangle\leq\sum_{i\in I}\langle T,T_{i}K^{\ast}\rangle\langle T_{i}K^{\ast},T\rangle\leq B\|K\|^{2}\langle T,T\rangle, \forall T\in End_{\mathcal{A}}^{\ast}(\mathcal{H, K}).
	\end{equation*}
	Then $\{T_{i}K^{\ast}\}_{i\in I}$ is a $K$-frame for $End_{\mathcal{A}}^{\ast}(\mathcal{H, K})$.
\end{proof}

\begin{Th}
	Let $K_{1}, K_{2}\in End_{\mathcal{A}}^{\ast}(\mathcal{H})$ and $\{T_{i}\in End_{\mathcal{A}}^{\ast}(\mathcal{H, K}): i\in I\}$ be a $K_{1}$-frame for $End_{\mathcal{A}}^{\ast}(\mathcal{H, K})$. Then $\{T_{i}K_{2}^{\ast}\}_{i\in I}$ is a $K_{2}K_{1}$-frame for $End_{\mathcal{A}}^{\ast}(\mathcal{H, K})$.
\end{Th}
\begin{proof}
	Let $\{T_{i}\}_{i\in I}$ be a $K_{1}$-frame for $End_{\mathcal{A}}^{\ast}(\mathcal{H, K})$.
	
	Then there exists two positive constants $A$ and $B$, such that
	\begin{equation*}
	A\langle TK_{1},TK_{1}\rangle\leq\sum_{i\in I}\langle T,T_{i}\rangle\langle T_{i},T\rangle\leq B\langle T,T\rangle, \forall T\in End_{\mathcal{A}}^{\ast}(\mathcal{H, K}).
	\end{equation*}
	Hence
	\begin{equation*}
	A\langle TK_{2}K_{1},TK_{2}K_{1}\rangle\leq\sum_{i\in I}\langle TK_{2},T_{i}\rangle\langle T_{i},TK_{2}\rangle\leq B\langle TK_{2},TK_{2}\rangle, \forall T\in End_{\mathcal{A}}^{\ast}(\mathcal{H, K}).
	\end{equation*}
	Thus
	\begin{equation*}
	A\langle TK_{2}K_{1},TK_{2}K_{1}\rangle\leq\sum_{i\in I}\langle T,T_{i}K_{2}^{\ast}\rangle\langle T_{i}K_{2}^{\ast},T\rangle\leq B\|K_{2}\|^{2}\langle T,T\rangle, \forall T\in End_{\mathcal{A}}^{\ast}(\mathcal{H, K}).
	\end{equation*}
	Then $\{T_{i}K_{2}^{\ast}\}_{i\in I}$ is a $K_{2}K_{1}$-frame for $End_{\mathcal{A}}^{\ast}(\mathcal{H, K})$.
\end{proof}
\begin{Cor}
	Let $K\in End_{\mathcal{A}}^{\ast}(\mathcal{H})$ and $\{T_{i}\in End_{\mathcal{A}}^{\ast}(\mathcal{H, K}): i\in I\}$ be a $K$-frame for $End_{\mathcal{A}}^{\ast}(\mathcal{H, K})$. Then $\{T_{i}(K^{\ast})^{N}\}_{i\in I}$ is a $K^{N+1}$-frame for $End_{\mathcal{A}}^{\ast}(\mathcal{H, K})$.	
\end{Cor}
\begin{proof}
	It follows from the previous theorem.
\end{proof}	
\begin{Th}
	A sequence $\{T_{i}\in End_{A}^{\ast}(\mathcal{H,K}): i\in I\}$ is a $K$-frame for $End_{A}^{\ast}(\mathcal{H,K})$ if and only if it is a $K$-g-frame for $\mathcal{H}$ whit respect to $\mathcal{K}$.
\end{Th}
\begin{proof}
	Let $\{T_{i}\}_{i\in I}$ be a $K$-g-frame for $\mathcal{H}$ with respect to $\mathcal{K}$.
	
	Then there exists two positive constants $A$ and $B$, such that 
	\begin{equation*}
	A\langle K^{\ast}x, K^{\ast}x\rangle_{\mathcal{A}}\leq\sum_{i\in I}\langle T_{i}x, T_{i}x\rangle_{\mathcal{A}}\leq B\langle x, x\rangle_{\mathcal{A}}, \forall x\in\mathcal{H},
	\end{equation*}
	 \begin{equation*}\Longleftrightarrow
	A\langle KK^{\ast}x,x\rangle_{\mathcal{A}}\leq\sum_{i\in I}\langle T_{i}^{\ast}T_{i}x,x\rangle_{\mathcal{A}}\leq B\langle x,x\rangle_{\mathcal{A}}, \forall x\in\mathcal{H}.
	\end{equation*}
	So
	\begin{equation*}
	AKK^{\ast}\leq\sum_{i\in I}T_{i}^{\ast}T_{i}\leq BI_{\mathcal{H}}.
	\end{equation*}
	Hence
	\begin{equation*}
	ATKK^{\ast}T^{\ast}\leq\sum_{i\in I}TT_{i}^{\ast}T_{i}T^{\ast}\leq BTT^{\ast}, \forall T\in End_{\mathcal{A}}^{\ast}(\mathcal{H, K}).
	\end{equation*}
	Thus
	\begin{equation*}
	A\langle TK,TK\rangle\leq\sum_{i\in I}\langle T,T_{i}\rangle\langle T_{i},T\rangle\leq B\langle T,T\rangle, \forall T\in End_{\mathcal{A}}^{\ast}(\mathcal{H, K}),
	\end{equation*}
	i.e $\{T_{i}\}_{i\in I}$ is a $K$-frame for $End_{\mathcal{A}}^{\ast}(\mathcal{H, K})$.
	
	Conversely, assume that $\{T_{i}\}_{i\in I}$ be a $K$-frame for $End_{\mathcal{A}}^{\ast}(\mathcal{H, K})$.
	
	Then there exists two positive constants $A$ and $B$, such that
	\begin{equation*}
	A\langle TK,TK\rangle\leq\sum_{i\in I}\langle T,T_{i}\rangle\langle T_{i},T\rangle\leq B\langle T,T\rangle, \forall T\in End_{\mathcal{A}}^{\ast}(\mathcal{H, K}),
	\end{equation*}
	
	\begin{equation*}\Longleftrightarrow
	ATKK^{\ast}T^{\ast}\leq\sum_{i\in I}TT_{i}^{\ast}T_{i}T^{\ast}\leq BTT^{\ast}, \forall T\in End_{\mathcal{A}}^{\ast}(\mathcal{H, K}).
	\end{equation*}
	So
	\begin{equation*}
	A\langle TKK^{\ast}T^{\ast}x,x\rangle_{\mathcal{A}}\leq\sum_{i\in I}\langle TT_{i}^{\ast}T_{i}T^{\ast}x,x\rangle_{\mathcal{A}}\leq B\langle TT^{\ast}x,x\rangle_{\mathcal{A}}, \forall T\in End_{\mathcal{A}}^{\ast}(\mathcal{H, K}), \forall x\in\mathcal{K}.
	\end{equation*}
	Let $y\in\mathcal{H}$ and $T\in End_{\mathcal{A}}^{\ast}(\mathcal{H, K})$ such that $T^{\ast}x=y$, then
	\begin{equation*}
	A\langle KK^{\ast}y,y\rangle_{\mathcal{A}}\leq\sum_{i\in I}\langle T_{i}^{\ast}T_{i}y,y\rangle_{\mathcal{A}}\leq B\langle y,y\rangle_{\mathcal{A}}, \forall y\in\mathcal{H},
	\end{equation*}
	i.e
	\begin{equation*}
	A\langle K^{\ast}y,K^{\ast}y\rangle_{\mathcal{A}}\leq\sum_{i\in I}\langle T_{i}y,T_{i}y\rangle_{\mathcal{A}}\leq B\langle y,y\rangle_{\mathcal{A}}, \forall y\in\mathcal{H},
	\end{equation*}
	thus  $\{T_{i}\}_{i\in I}$ is a $K$-g-frame for $\mathcal{H}$ with respect to $\mathcal{K}$.
\end{proof}
\begin{Cor}
	A sequence $\{T_{i}\in End_{A}^{\ast}(\mathcal{H, K}): i\in I\}$ is a tight $K$-frame for $End_{A}^{\ast}(\mathcal{H, K})$ if and only if it is a tight $K$-g-frame for $\mathcal{H}$ whit respect to $\mathcal{K}$.
\end{Cor}
\section{generalized K-frame for $End_{A}^{\ast}(\mathcal{H, K})$}
We begin this section with the follwing definition.
\begin{Def}
	Let $K\in End_{A}^{\ast}(\mathcal{H})$. A sequence $\{T_{i}\in End_{A}^{\ast}(\mathcal{H, K}): i\in I\}$ is said to be a generalized $K$-frame for $End_{A}^{\ast}(\mathcal{H, K})$ if there exist strictly positive elements $A$, $B$ in $\mathcal{A}$ sach that
	\begin{equation}
	A\langle TK,TK\rangle\leq\sum_{i\in I}\langle T,T_{i}\rangle\langle T_{i},T\rangle\leq B\langle T,T\rangle, \forall T\in End_{A}^{\ast}(\mathcal{H,K}),
	\end{equation}
	where the series converges in the strong operator topology.
\end{Def}
In the follwing $\mathcal{A}$ is a commutative $C^{\ast}$-algebra.
\begin{Th}
	A sequence $\{T_{i}\in End_{A}^{\ast}(\mathcal{H,K}): i\in I\}$ is a generalized $K$-frame for $End_{A}^{\ast}(\mathcal{H,K})$ if and only if it is an $\ast$-$K$-g-frame for $\mathcal{H}$ whit respect to $\mathcal{K}$.
\end{Th}
\begin{proof}
	Let $\{T_{i}\}_{i\in I}$ be an $\ast$-$K$-g-frame for $\mathcal{H}$ with respect to $\mathcal{K}$.
	
	Then there exist strictly positive elements $A$ and $B$ in $\mathcal{A}$, such that 
	\begin{equation*}
	A\langle K^{\ast}x, K^{\ast}x\rangle_{\mathcal{A}}\leq\sum_{i\in I}\langle T_{i}x, T_{i}x\rangle_{\mathcal{A}}\leq B\langle x, x\rangle_{\mathcal{A}}, \forall x\in\mathcal{H},
	\end{equation*}
	 \begin{equation*}\Longleftrightarrow
	A\langle KK^{\ast}x,x\rangle_{\mathcal{A}}\leq\sum_{i\in I}\langle T_{i}^{\ast}T_{i}x,x\rangle_{\mathcal{A}}\leq B\langle x,x\rangle_{\mathcal{A}}, \forall x\in\mathcal{H}.
	\end{equation*}
	So
	\begin{equation*}
	AKK^{\ast}\leq\sum_{i\in I}T_{i}^{\ast}T_{i}\leq BI_{\mathcal{H}}.
	\end{equation*}
	Hence
	\begin{equation*}
	ATKK^{\ast}T^{\ast}\leq\sum_{i\in I}TT_{i}^{\ast}T_{i}T^{\ast}\leq BTT^{\ast}, \forall T\in End_{\mathcal{A}}^{\ast}(\mathcal{H, K}).
	\end{equation*}
	Thus
	\begin{equation*}
	A\langle TK,TK\rangle\leq\sum_{i\in I}\langle T,T_{i}\rangle\langle T_{i},T\rangle\leq B\langle T,T\rangle, \forall T\in End_{\mathcal{A}}^{\ast}(\mathcal{H, K}),
	\end{equation*}
	i.e $\{T_{i}\}_{i\in I}$ is a generalized $K$-frame for $End_{\mathcal{A}}^{\ast}(\mathcal{H, K})$.
	
	Conversely, assume that $\{T_{i}\}_{i\in I}$ be a generalized $K$-frame for $End_{\mathcal{A}}^{\ast}(\mathcal{H, K})$.
	
	Then there exist strictly positive elements $A$ and $B$ in $\mathcal{A}$, such that
	\begin{equation*}
	A\langle TK,TK\rangle\leq\sum_{i\in I}\langle T,T_{i}\rangle\langle T_{i},T\rangle\leq B\langle T,T\rangle, \forall T\in End_{\mathcal{A}}^{\ast}(\mathcal{H, K}),
	\end{equation*}
	
	\begin{equation*}\Longleftrightarrow
	ATKK^{\ast}T^{\ast}\leq\sum_{i\in I}TT_{i}^{\ast}T_{i}T^{\ast}\leq BTT^{\ast}, \forall T\in End_{\mathcal{A}}^{\ast}(\mathcal{H, K}).
	\end{equation*}
	So
	\begin{equation*}
	A\langle TKK^{\ast}T^{\ast}x,x\rangle_{\mathcal{A}}\leq\sum_{i\in I}\langle TT_{i}^{\ast}T_{i}T^{\ast}x,x\rangle_{\mathcal{A}}\leq B\langle TT^{\ast}x,x\rangle_{\mathcal{A}}, \forall T\in End_{\mathcal{A}}^{\ast}(\mathcal{H, K}), \forall x\in\mathcal{K}.
	\end{equation*}
	Let $y\in\mathcal{H}$ and $T\in End_{\mathcal{A}}^{\ast}(\mathcal{H, K})$ such that $T^{\ast}x=y$, then
	\begin{equation*}
	A\langle KK^{\ast}y,y\rangle_{\mathcal{A}}\leq\sum_{i\in I}\langle T_{i}^{\ast}T_{i}y,y\rangle_{\mathcal{A}}\leq B\langle y,y\rangle_{\mathcal{A}}, \forall y\in\mathcal{H},
	\end{equation*}
	i.e
	\begin{equation*}
	A\langle K^{\ast}y,K^{\ast}y\rangle_{\mathcal{A}}\leq\sum_{i\in I}\langle T_{i}y,T_{i}y\rangle_{\mathcal{A}}\leq B\langle y,y\rangle_{\mathcal{A}}, \forall y\in\mathcal{H},
	\end{equation*}
	thus  $\{T_{i}\}_{i\in I}$ is an $\ast$-$K$-g-frame for $\mathcal{H}$ with respect to $\mathcal{K}$.
\end{proof}
\begin{Cor}
	A sequence $\{T_{i}\in End_{A}^{\ast}(\mathcal{H, K}): i\in I\}$ is a generalized tight $K$-frame for $End_{A}^{\ast}(\mathcal{H, K})$ if and only if it is a tight $\ast$-$K$-g-frame for $\mathcal{H}$ whit respect to $\mathcal{K}$.
\end{Cor}

\end{document}